\documentclass[11pt,amsfonts]{article}
\usepackage{graphicx}
\usepackage{latexsym}
\usepackage{amssymb}
\usepackage{amsmath}
\usepackage{enumerate}

\usepackage{layout}
\usepackage{eufrak}
\newtheorem{prop}{Proposition}
\newtheorem{lemma}{Lemma}

\newtheorem{corollary}{Corollary}
\newtheorem{theorem}{Theorem}

\newtheorem{remark}{Remark}
\newtheorem{example}{Example}

\newcommand {\qed}%
{%
    {}\hfill
    {}\hfill
    {$\square $}%
    \vspace {0.3cm}%
    \pagebreak [2]%
    \par
}%
\newenvironment{proof}[1]{%
    \vspace{0.3cm}%
    \pagebreak [2]%
    \par%
    \noindent{\bf Proof~#1\ }}{\qed }%

\def\real{{\mathord{{\rm I\kern-2.8pt R}}}}        
\def\inte{{\mathord{{\rm I\kern-2.8pt N}}}}

\def\sZZ{{\rm Z\kern-2.8ptem{}Z}}

\def\z{{\mathchoice
  {\sZZ}
  {\sZZ}
  {\rm Z\kern-0.30em{}Z}
  {\rm Z\kern-0.25em{}Z} }}
\def\sQQ{{\kern 0.27em \vrule height1.45ex width0.03em depth0em
          \kern-0.30em \rm Q}}
\def\qu{{\mathchoice
    {\sQQ}
    {\sQQ}
  {\kern 0.225em \vrule height1.05ex width0.025em depth0em \kern-0.25em \rm Q}
  {\kern 0.180em \vrule height0.78ex width0.020em depth0em \kern-0.20em \rm Q}
        }}
\def\sCC{{\kern 0.27em \vrule height1.45ex width0.03em depth0em
          \kern-0.30em \rm C}}
\def\complex{{\mathchoice
    {\sCC}
    {\sCC}
  {\kern 0.225em \vrule height1.05ex width0.025em depth0em \kern-0.25em \rm C}
  {\kern 0.180em \vrule height0.78ex width0.020em depth0em \kern-0.20em \rm C}
        }}

\newcommand{\R}{\mathbb{R}}

\newcommand{\ba}{\begin{array}}
\newcommand{\ea}{\end{array}}
\newcommand{\be}{\begin{equation}}
\newcommand{\ee}{\end{equation}}
\newcommand{\bea}{\begin{eqnarray}}
\newcommand{\eea}{\end{eqnarray}}
\newcommand{\beaa}{\begin{eqnarray*}}
\newcommand{\eeaa}{\end{eqnarray*}}

\def\b{\beta}

\def\z{\zeta}

\font\tenmath=msbm10 \font\sevenmath=msbm7 \font\fivemath=msbm5
\newfam\mathfam \textfont\mathfam=\tenmath
\scriptfont\mathfam=\sevenmath \scriptscriptfont\mathfam=\fivemath

\def \b{\noindent}

\def \={{\buildrel {\rm (law)} \over =}}
\def \E{{\mathbb E}}

\def\cB{{\cal B}}

\def\cF{{\cal F}}

\def\cS{{\cal S}}

\def\qed{ \hfill \vrule width.25cm height.25cm depth0cm\smallskip}

\newcommand{\basa}{\begin{assumption}}
\newcommand{\easa}{\end{assumption}}

\newcommand{\bas}{\begin{assum}}
\newcommand{\eas}{\end{assum}}

\def\limsup{\mathop{\overline{\rm lim}}}
\def\liminf{\mathop{\underline{\rm lim}}}

\newcommand{\ignore}[1]{}
\begin{document}

\renewcommand{\thefootnote}{\fnsymbol{footnote}}

\renewcommand{\thefootnote}{\fnsymbol{footnote}}

\title{Sample Paths of the Solution to the Fractional-colored Stochastic Heat Equation}

\author{Ciprian A. Tudor  \footnote{  Supported by the CNCS grant PN-II-ID-PCCE-2011-2-0015 (Romania).
} $^{1, 2} $ and
Yimin Xiao \thanks{Research of Y. Xiao is partially
supported by NSF grants DMS-1307470 and DMS-1309856.} $^{3}$, \vspace*{0.1in} \\
 $^{1}$ Laboratoire Paul Painlev\'e, Universit\'e de Lille 1\\
 F-59655 Villeneuve d'Ascq, France.\\
$^{2}$  Academy of Economical Studies, Piata Romana nr.6, sector 1 \\
 Bucharest, Romania.\\
 \quad tudor@math.univ-lille1.fr\\
$^{3}$  Department of Statistics and Probability, Michigan State University\\
 East Lansing, MI 48824, U.S.A.\\
 xiao@stt.msu.edu\\
\vspace*{0.1in} }
\maketitle

\maketitle

\begin{abstract}
Let $u = \{u(t,x), t\in [0,T], x\in \mathbb{R} ^ {d}\}$ be the solution to the linear stochastic heat equation 
driven by a fractional noise in time with correlated spatial structure. We study various  path properties 
of the process $u$ with respect to the time and space variable, respectively.  In particular, we derive 
their exact uniform and local moduli of continuity and Chung-type laws of the iterated logarithm.
\end{abstract}

\vskip0.3cm

{\bf 2010 AMS Classification Numbers:}  60H55, 60H07, 91G70.

 \vskip0.3cm


\section{Introduction}

Stochastic analysis of fractional Brownian motion (fBm)
naturally led to the study of stochastic partial differential equations (SPDEs) driven by this
Gaussian process. The motivation comes from wide applications of fBm. We refer,
among others, to \cite{GLT06}, \cite{maslowski-nualart03}, \cite{nualart-vuillermont06},
\cite{QS-tindel07} and \cite{TTV} for theoretical studies of SPDEs driven by fBm. To list only a
few examples of applications of fractional noises in various  areas, we mention
\cite{KouSinney04} for biophysics, \cite{BPS04} for financial time series, \cite{DMS03} for electrical
engineering, and \cite{CCL03} for physics.

The purpose  of our present paper is to study fine properties of the solution to the stochastic 
heat equation driven by a Gaussian noise which is fractional in time and colored in space. 
Our work continues, in part, the line of research which concerns SPDEs driven by the
fBm but at the same time it follows the research line initiated by Dalang \cite{Da} which treats
equations with white noise  in time and non trivial correlation in space. More precisely, we consider
a stochastic linear equation driven by a Gaussian noise that behaves as a fractional Brownian motion
with respect to its time variable and has a correlated spatial structure. A necessary and sufficient
condition for the existence of the solution has been given in \cite{BalanTudor1} and other results 
has been given in \cite{T}, \cite{EF}, \cite{B}, \cite{BC} among others.

To briefly describe the context, consider the stochastic heat equation with additive
noise
\begin{eqnarray}
\label{eq1} \frac{\partial u}{\partial t} &=& \frac{1}{2}
\Delta u + \dot {W} , \quad t  \in [0,T], \mbox{\ } x\in \mathbb{R}^{d}, \\
\nonumber u (0,x) &=& 0, \quad x \in \mathbb{R}^d,
\end{eqnarray}
where $T > 0$ is a constant and $\dot {W}$ is usually referred
to as a {\it space-time white noise}. It is well-known (see for example the now classical paper  \cite{Da}) that (\ref{eq1})
admits a unique mild solution if and only if $d=1$. This mild solution is defined as
\begin{equation}
\label{sol}
 u(t,x)=\int_{0}^{t}
\int_{\mathbb{R}^{d}} G(t-s,x-y)W(ds,dy), \hskip0.5cm t\in [0,T], \, x\in \mathbb{R}.
\end{equation}
In the above,  $W=\{W(t,A); t \ge 0,\, A \in \cB_{b}(\mathbb{R}^d)\}$ is
a zero-mean Gaussian process with covariance given by
\begin{equation}
\label{cov1a}
\E \left(  W(t,A)W(s,B) \right) = (s\wedge t) \lambda (A \cap B),
\end{equation}
where $\lambda$ denotes the Lebesque measure and  $\cB_{b}(\mathbb{R}^d)$ is the collection of
all bounded Borel subsets of $\mathbb{R}^d$. The integral  in (\ref{sol}) is a Wiener integral with 
respect to the Gaussian process $W$  and $G $ is the Green kernel of the heat equation given by
\begin{equation}
\label{fund-sol-heat} G(t,x)=\left\{
\begin{array}{ll} (2 \pi t)^{-d/2} \exp\left( -\frac{|x|^{2}}{2t}\right) & \mbox{if $t>0, x \in \mathbb{R}^{d}$},  \\
0 & \mbox{if $t \leq 0, x \in \mathbb{R}^{d}$}.
\end{array} \right.
\end{equation}
Consequently the mild solution $\{u(t,x), t\in [0,T] , x \in \mathbb{R}\} $ is a centered
two-parameter Gaussian process (also called a Gaussian random field). It is well-known (see e.g.
\cite{Swanson1} or \cite{T}) that
the solution defined by the (\ref{sol}) is  well-defined  if and only if $d=1$ and in this case   
the covariance of the solution (\ref{sol}), when $x\in \mathbb{R}$ is fixed, satisfies   
\begin{equation}\label{covu}
\E \left(  u(t,x) u(s,x) \right) = \frac{1}{\sqrt{2\pi} } \left( \sqrt{ t+s}-\sqrt{\vert t-s\vert } \right),
\mbox{ for every }  s,t \in [0,T].
\end{equation}
This establishes an interesting connection between the law of the solution (\ref{sol})
and the so-called {\it bifractional Brownian motion} introduced by \cite{HV}. Recall that, given
constants $H\in (0,1)$ and $K\in (0,1]$,  the bifractional Brownian
motion $(B^{H,K}_{t})_{t \in [0,T]}$ is a centered Gaussian process with covariance
  \begin{equation}
    \label{cov-bi}
    R^{H,K}(t,s) := R(t,s)= \frac{1}{2^{K}}\left( \left(
        t^{2H}+s^{2H}\right)^{K} -\vert t-s \vert ^{2HK}\right), \hskip0.5cm s,t \in [0,T].
  \end{equation}
Relation (\ref{covu}) implies that, when $x \in \R$ is fixed, the Gaussian process
$\left\{u(t,x), t \in [0,T] \right\} $ is  a bifractional Brownian motion with
parameters $H=K=\frac{1}{2}$ multiplied by the constant $2^{-K} \frac{1}{\sqrt{2\pi }}$.
Therefore, many sample path properties of the solution to the heat equation driven by
space time white-noise follow from \cite{RT, LN09, RT, TX}.
\vskip0.3cm

In order to avoid the restriction to the dimension $d=1$, several authors considered other 
types of noises, usually more regular than the time-space white noise.   An approach is to 
consider the {\it white-colored
noise}, meaning  a Gaussian process $W=\{W(t,A); t \ge 0,\, A \in \cB_{b}(\mathbb{R}^d)\}$
with covariance with zero mean and covariance
\begin{equation}\label{cov2a}
\E \left(W(t,A) W(s,B) \right)= (t\wedge s) \int_{A}\int_{B} f(z-z') dzdz'.
\end{equation}
Here  the kernel $f$ is  the Fourier transform of a tempered  non-negative measure
 $\mu$ on $\mathbb{R}^d$, i.e. $\mu$ is a non-negative measure which satisfies:
\begin{equation} \label{mu-tempered}
\int_{\mathbb{R}^d}
\left(\frac{1}{1+ |\xi|^{2}} \right)^l \mu(d\xi)<\infty \quad
 \mbox{for some} \ l >0.
\end{equation}
Recall that the Fourier transform $f$ of $\mu$ is defined through
$$\int_{\mathbb{R}^d}f(x)\varphi(x)dx=\int_{\mathbb{R}^d}\cF
\varphi(\xi)\mu(d\xi), \quad \forall \varphi \in \cS(\mathbb{R}^d), $$
where $\cS(\mathbb{R}^d)$ denotes the Schwarz space on $\mathbb{R} ^{d}$. 
A useful connection between the kernel $f$ and its associated measure $\mu$  is the
Parseval inequality:  for any
$\varphi, \psi \in \cS(\mathbb{R}^d)$,
\begin{equation}\label{parseval}
\int_{\mathbb{R}^d} \int_{\mathbb{R}^d} \varphi(x)f(x-y)\psi(y)dx dy=(2\pi) ^{-d}
\int_{\mathbb{R}^d}\cF \varphi(\xi) \overline{\cF \psi(\xi)}\mu(d\xi).
\end{equation}

In this work, we will include the following two following basic examples.

\begin{example}
The Riesz kernel of order $\alpha$: 
\begin{equation}
\label{riesz}
f(x)=R_{\alpha}(x):=
\gamma_{\alpha,d}|x|^{-d+\alpha},
\end{equation}
 where  $0<\alpha<d$, $\gamma_{\alpha,d}=2^{d-\alpha} \pi^{d/2}
 \Gamma((d-\alpha)/2)/\Gamma(\alpha/2)$. In this case,
$\mu(d\xi)=  |\xi|^{-\alpha}d \xi. $
\end{example}

\begin{example}
\label{bessel} The Bessel kernel of order $\alpha$:
$$f(x)=B_{\alpha}(x):=\gamma'_{\alpha}\int_{0}^{\infty}w^{(\alpha-d)/2-1}
e^{-w}e^{-|x|^{2}/(4w)} dw, $$ where $\alpha>0,$\,
$\gamma'_{\alpha}=(4\pi)^{\alpha/2}\Gamma(\alpha/2)$. In this
case, $\mu(d\xi)=(1+|\xi|^{2})^{-\alpha/2} d \xi.$
\end{example}

For any function  $g \in L^{1} (\mathbb{R} ^{d})$ we denote by $\cF g$  the Fourier transform of $g$, i.e.,
\begin{equation*}
(\cF g) (\xi) = \int_{\mathbb{R} ^{d} } e^{-i \langle \xi, x \rangle} g(x)dx, \hskip0.5cm \xi \in \mathbb{R}^{d},
\end{equation*}
where $\langle \cdot, \cdot\rangle$ denotes the inner product in $\R^d$. 
Recall that for any  $ t \in \R$ and $x \in \R^d$, the Fourier transform of the fundamental solution (\ref{fund-sol-heat}) is given by
 \begin{equation}
\label{Fourier-heat} \cF G (t, x-\cdot)(\xi)= \exp
\left( i  \langle x, \xi \rangle -\frac{t|\xi|^2}{2} \right) \mathbf{1}_{\{t > 0\}} (\xi), \hskip0.4cm \xi \in \mathbb{R}^d.
\end{equation}

Dalang \cite{Da} proved that the stochastic heat equation with
white-colored noise admits a unique solution if and only if 
\begin{equation}\label{Eq:Dalang}
\int_{\mathbb{R}^d} \frac{1}{1+ |\xi|^{2}}  \mu(d\xi)<\infty.
\end{equation}
Obviously, this condition allows $x$ to be in higher dimensional space. For example in
the case of the Riesz kernel, the stochastic heat equation with white-colored noise
admits an unique solution if and only if $d<2+ \alpha.$

Under (\ref{Eq:Dalang}), the solution of (\ref{eq1}) with white-colored noise can still be written as in (\ref{sol}).
One can compute the covariance of the solution with respect to the time variable (see e.g. \cite{OT}).
For fixed $x\in \mathbb{R}^{d}$, and for every $s\leq t$ it follows from
Parseval's identity and (\ref{Fourier-heat})
\begin{eqnarray}
\label{Eq:OT}
\E \big(u(t,x) u(s,x) \big)=  (2\pi) ^{-d}  \int_{0}^{s} du \int_{\mathbb{R} ^{d}}\mu (d\xi)
e^{-\frac{1}{2} (t-u) \vert \xi \vert ^{2}}e^{-\frac{1}{2} (s-u) \vert \xi \vert ^{2}}.
\end{eqnarray}
In the case when $f$ is the Riesz kernel (i.e., $\mu (d\xi)= \vert \xi \vert ^{-\alpha}d\xi$,
we get
\begin{eqnarray*}
\E \big(u(t,x) u(s,x)\big) =(2\pi )^{-d} \int_{\mathbb{R} ^{d}} \frac{d\xi} {\vert \xi \vert ^{\alpha}}\,
e^{-\frac{1}{2}  \vert \xi \vert ^{2}}\frac{1}{1-\frac{d -\alpha}{2}}
\left((t+s)^{1-\frac{d-\alpha }{2}}-(t-s)^{1-\frac{d-\alpha }{2}}\right).
\end{eqnarray*}
In this case,  the solution coincides, modulo a constant, with a
bifractional Brownian motion with parameters $H= \frac{1}{2}$ and $K= 1-\frac{d-\alpha }{2}$.
Thus the sample path properties of this process can be  deduced from \cite{LN09, RT, TX}.

Our purpose in this paper is to study  the linear heat equation driven by a fractional-colored  
Gaussian noise  (see section 2 for details). When the structure of 
the noise with respect to the time variable changes and the white noise is 
replaced by a fractional noise,  the solution does not coincide with a bifractional 
Brownian motion anymore  (see \cite{BoTu} or \cite{T}). Some new methods are needed for  
analyzing the path properties of the solution with respect to the time and to the space variables. 
By appealing to general methods for Gaussian processes and fields (cf. e.g., \cite{MRbook, Xiao07, Xiao1}), 
we establish the exact uniform and local moduli of continuity and Chung type laws of the iterated logarithm
for the solution of (\ref{eq1}) with a fractional-colored  Gaussian noise.

Throughout this paper, we will use C to denote unspecified positive finite constants which may
be different in each appearance. More specific constants are numbered as $c_1,\, c_2 , \ldots.$

\section{The solution to the stochastic heat equation with fractional-colored noise}

We consider a Gaussian field
$\left\{ W^{H} (t,A), t\in [0,T] , A\in \cB_b (\mathbb{R} ^{d}) \right\} $ with covariance
\begin{equation}
\label{cov2}
\E \left( W^{H} (t,A) W^{H} (s, B) \right) = R_{H}(t,s) \int_{A}\int_{B} f(z-z') dzdz',
\end{equation}
where $R_{H}(t,s):=\frac{1}{2} (t^{2H}+ s^{2H} -\vert t-s \vert ^{2H} )$ is the covariance of
the fractional Brownian motion with index $H$ and $f$ is the Fourier transform of
a tempered measure $\mu$.  This noise is usually
called {\it fractional-colored } noise.  We will assume throughout the paper that the Hurst parameter
$H$ is contained in the interval $(\frac{1}{2}, 1)$. 

Consider the linear stochastic heat equation
\begin{eqnarray}
\label{heatEq}
\frac{\partial u}{\partial t} = \frac{1}{2}\Delta u + \dot{W}^{H}, \mbox{\  \  \  }
t \in \left[0,T\right], x \in \mathbb{R}^{d}
\end{eqnarray}
with vanishing initial condition, where $\{W^H(t,x), \, t \in \left[0,T\right], x \in \mathbb{R}^{d} \}$ is a
centered Gaussian noise with covariance (\ref{cov2}). In the following we collect some known facts:
\begin{description}
\item{$\bullet$ } A necessary and sufficient condition for the existence of the mild solution to the
fractional-colored heat equation (\ref{heatEq})  has been given in  \cite{BalanTudor1} (see also \cite{OT}).
Namely, (\ref{heatEq})  has a solution  $\{u (t,x), t \geq 0, x \in \mathbb{R}^{d}\}$ that satisfies
$$ \underset{t \in \left[0,T\right], x \in \mathbb{R}^{d}}{\sup} \E\left(u(t,x)^{2}\right) < +\infty $$
if and only if
\begin{equation}\label{29m-1}
\int_{\mathbb{R}^d}
\left(\frac{1}{1+ |\xi|^{2}} \right)^{2H} \mu(d\xi)<\infty.
\end{equation}

\item{$\bullet$ } When (\ref{29m-1}) holds,  the solution to (\ref{heatEq}) can be written in the mild form as
\begin{equation}\label{u}
u(t,x)= \int_{0} ^{t} \int_{\mathbb{R} ^{d}} G(t-u, x-y) W^{H}(du, dy), \hskip0.5cm t\in [0,T],\
x \in \mathbb{R} ^{d} .
\end{equation}
It follows from \cite{BalanTudor1} or \cite{OT} that, for $x\in \mathbb{R}^{d}$ fixed, the covariance 
function of the Gaussian process $\{u(t,x), t\in [0,T]\}$
can be written as 
\begin{equation}
\label{cov1}
\E \left(u(t,x)u(s,x) \right) = \alpha _{H} (2\pi ) ^{-d} \int_{0} ^{t} \int_{0} ^{s} \frac{dudv}
{ \vert u-v\vert ^{2-2H}} \int_{\mathbb{R} ^{d}}   \mu (d\xi) e^{-\frac{(t-u) \vert \xi \vert ^{2}}{2}}
e^{-\frac{(s-v) \vert \xi \vert ^{2}}{2}}
\end{equation}
with $\alpha_{H}= H(2H-1)$. In the particular case where the spatial covariance is given by the
Riesz kernel, the process
$t \mapsto u(t, x)$ is self-similar with parameter $H-\frac{d-\alpha}{4}.$ However, for $H \in (\frac 1 2, 1),$
$\{u(t, x), t \ge 0\}$ is no longer a bifractional Brownian motion.
\end{description}

For $0<\alpha <d$,  the notation 
\begin{equation}
\label{hyp}
 \mu (d \xi) \asymp  \vert \xi  \vert ^{-\alpha}d\xi, 
\end{equation}
means that for every non-negative function $h$  there exist two  positive 
and finite constants $C$ and $C'$, which may depend on $h$,  such that
\begin{equation}\label{Eq:mu}
C'\int_{\mathbb{R} ^{d} } h(\xi )   \vert \xi  \vert ^{-\alpha}d\xi\leq \int_{\mathbb{R} ^{d}} h(\xi) \mu (d\xi )
\leq C \int_{\mathbb{R} ^{d} } h(\xi )   \vert \xi  \vert ^{-\alpha}d\xi.
\end{equation}
It has been proven in \cite{T} that, under condition (\ref{hyp}), there exist two strictly positive constants
$c_{1}, c_{2}$ such that for all $t,s \in [0,1]$ and for all $x\in \mathbb{R}^{d}$,
\begin{equation}
\label{reg}
c_{1} \vert t-s \vert ^{2H-\frac{d-\alpha}{2}} \leq \E
\left(| u(t,x) -u(s, x)| ^{2}\right) \leq c_{2} \vert t-s \vert
^{2H-\frac{d-\alpha}{2}}.
\end{equation}

\begin{remark}
The Riesz kernel obviously satisfies (\ref{hyp}). The Bessel kernel satisfies (\ref{hyp}) and the constants in 
\eqref{Eq:mu} are  $C=1$ and $C'>0$ depending on $h$.

Under Condition (\ref{hyp}),  the condition (\ref{29m-1}) is equivalent to
\begin{equation}\label{cc}
d<4H+ \alpha.
\end{equation}
\end{remark}

In the sequel (e.g., Theorem \ref{deco} below), we will also use the notation $f(x) \asymp g(x)$ which means that there
exists two   positive and finite constants $C$ and $C'$ such that $C' \le f(x)/g(x)\le C$ for all $x$ in the domain of $f$ and $g$.
In the rest of  the paper, we will assume that (\ref{hyp}) is satisfied.

\subsection{Sharp regularity in time}
\label{sec:time}

For any fixed $x \in \R^d$, to further study the regularity properties of the sample function
$t \mapsto u(t, x)$,  we decompose the solution $\left\{u(t,x), t\geq 0\right\} $ in
(\ref{u}) into the sum of a Gaussian process $U = \{U(t), t \ge 0\}$,
which has stationary increments, and another Gaussian process whose sample functions 
are continuously differentiable on $(0, \infty)$.  In particular, when the spatial covariance 
of noise $ W^{H}$ is given by  the Riesz kernel (\ref{riesz}), then $U = \{U(t), t \ge 0\}$ 
becomes fractional Brownian motion with Hurst parameter $\gamma= H-\frac{d-\alpha}{4}$. 
We will apply this decomposition to obtain the regularity in time of the solution $\left\{u(t,x), t\geq 0\right\}$.

Motivated by \cite{MT02} (see also \cite{MZ, WX06} for related results), we introduce the pinned string 
process in time $\left\{  U(t), t\geq 0 \right\} $  defined by
\begin {eqnarray*}
U(t)&=& \int_{-\infty} ^{0} \int_{\mathbb{R} ^{d}} \big( G(t-u, x-y) -G (-u, x-y) \big) W ^{H}(du, dy)\\
&& \qquad \qquad \qquad + \int_{0} ^{t} \int_{\mathbb{R} ^{d} } G(t-u, x-u)W ^{H} (du, du).
\end{eqnarray*}
Note that $U(0) =0$ and  $U(t)$ can be expressed as
\begin{equation}\label{u1}
U(t)= \int_{\mathbb{R}} \int_{\mathbb{R} ^{d}} \big( G((t-u)_{+}, x-y) -G( (-u)_{+} , x-y) \big) W ^{H}(du, dy).
\end{equation}
In the above, $a_{+}= \max (a, 0)$. The following theorem shows that $\left\{  U(t), t\geq 0 \right\} $
has stationary increments and identifies its spectral measure. This information is useful for appleying the
methods in  \cite{MRbook, Xiao07, Xiao1} to study sample path properties of $\left\{ U(t), t\geq 0\right\}$.


\begin{theorem}\label{deco}
The Gaussian process $\left\{ U(t), t\geq 0\right\}$ given by (\ref{u1}) has stationary increments and its
spectral density is given by
\begin{equation}\label{Eq:spectralden1}
  f_{U}(\tau) = \frac{2 (2 \pi)^{-d}\alpha _{H}} {  \vert \tau \vert ^{2H-1} } \,
\int_{\mathbb{R} ^{d}} \frac{\mu(d\xi)} { \tau ^{2} + \frac{\vert \xi \vert ^{4}}{4}}.
\end{equation}
Moreover, under condition (\ref{hyp}), we have
\begin{equation}\label{Eq:spectralden2}
f_{U}(\tau) \asymp \frac 1 {\vert \tau \vert ^{2H - \frac{d-\alpha}{2} +1}}.
\end{equation}
\end{theorem}

\begin{proof}:  For every $0\leq s<t$,  by Parseval's identity (\ref{parseval})  and relation (\ref{Fourier-heat}), we can write
\begin{eqnarray*}
&& \mathbb{E} \left[ (U(t)- U(s))^{2}\right] \\
&&= \mathbb{E} \left( \int_{\mathbb{R}} \int_{\mathbb{R} ^{d}} \left( G((t-u)_{+}, x-y) -G( (s-u)_{+} , x-y) \right) W ^{H}(du, dy)\right) ^{2} \\
&&= (2 \pi)^{-d} \alpha_{H}\int_{\mathbb{R}^d} \mu(d\xi) \int_{\mathbb{R}} \int_{\mathbb{R}} \frac{dudv} {\vert u-v\vert ^{2-2H}}
\left( e^{-\frac{1}{2} (t-u) \vert \xi \vert ^{2}} \mathbf{1}_{\{t >u\}}- e^{-\frac{1}{2} (s-u) \vert \xi \vert ^{2}} \mathbf{1}_{\{s >u\}}\right)\\
&&\qquad \qquad \qquad \qquad \qquad \qquad \qquad \qquad \times
\left( e^{-\frac{1}{2} (t-v)\vert \xi \vert ^{2}} \mathbf{1}_{\{t >v\}}- e^{-\frac{1}{2} (s-v) \vert \xi \vert ^{2}} \mathbf{1}_{\{s >v\}}\right).
\end{eqnarray*}
Let $\varphi(u) =  e^{-\frac{1}{2} (t-u) \vert \xi \vert ^{2}} \mathbf{1}_{\{t >u\}}- e^{-\frac{1}{2} (s-u) \vert \xi \vert ^{2}} \mathbf{1}_{\{s >u\}}$.
Then its Fourier transform is
\begin{equation}\label{Eq:F-phi}
{\mathcal F}\varphi(\tau)= \big(e^{-i t\tau} - e^{-i s\tau}\big) \frac 1 {i\tau + \frac 1 2 |\xi|^2}.
\end{equation}
 By using again Parseval's relation (\ref{parseval}) in dimension $d=1$ and (\ref{Eq:F-phi}), we get
\begin{eqnarray*}
\mathbb{E}\left[( U(t)- U(s))^{2}\right]
&=&(2 \pi)^{-d} \alpha _{H} \int_{\mathbb{R}^{d}} \mu(d\xi) \int_{\mathbb{R}}
\frac{d\tau} {\vert \tau \vert ^{2H-1}} \big| {\mathcal F}\varphi(\tau)\big|^{2}\\
&&= (2 \pi)^{-d} \alpha _{H}\int_{\mathbb{R}^{d}} \mu(d\xi) \int_{\mathbb{R} }\frac{ d\tau} {\vert \tau \vert ^{2H-1}}
\frac{ 2 \left[ 1-\cos ((t-s)\tau)\right] }{\tau ^{2} + \frac{\vert \xi \vert ^{4}}{4}}\\
&&=2(2 \pi)^{-d} \alpha _{H} \int_{\mathbb{R}}\big[1-\cos ((t-s)\tau)\big] \frac{ d\tau} {\vert \tau \vert ^{2H-1} } \,
\int_{\mathbb{R} ^{d}} \frac{\mu(d\xi)} { \tau ^{2} + \frac{\vert \xi \vert ^{4}}{4}},
\end{eqnarray*}
where the last step follows from Fubini's theorem and the convergence of the last integral in $\mu(d\xi)$ is
guaranteed by relation  (\ref{29m-1}). It follows that the Gaussian process $\left\{ U(t), t\geq 0\right\}$
has stationary increments and its spectral density  is  given by (\ref{Eq:spectralden1}).

Under condition (\ref{hyp}), we have
\[
f_{U}(\tau)\asymp \frac{1}{ \vert \tau \vert^{2H-1}}
\int_{\mathbb{R} ^{d}} \frac{d\xi } { \vert \xi \vert^{\alpha} \left( \tau ^{2} + \frac{ \vert \xi \vert ^{4}}{2} \right) }
\asymp \frac 1 {\vert \tau \vert ^{2H - \frac{d-\alpha}{2} +1}}.
\]
This finishes the proof of  (\ref{Eq:spectralden2}) and thus the conclusion of Theorem \ref{deco} is obtained.
\end{proof}

When the spatial covariance of noise $W^{H}$ is given by the Riesz kernel (\ref{riesz}),  we have
\begin{eqnarray*}
f_{U}(\tau)&=& \frac{\alpha _{H} 2 ^{2H-1} \Gamma (H-\frac{1}{2}) }{ (2 \pi)^{d +\frac 1 2}   \Gamma (1-H) \, \vert \tau \vert ^{2H-1}}
\int_{\mathbb{R} ^{d}} \frac{d\xi } { \vert \xi \vert ^{\alpha} \left( \tau ^{2} + \frac{ \vert \xi \vert ^{4}}{4} \right) }\\
&=& \frac{ \alpha _{H} 2 ^{2H-1} \Gamma (H-\frac{1}{2}) }{ (2 \pi)^{d + \frac 1 2}   \Gamma (1-H) \,  \vert \tau \vert ^{2H - \frac{d-\alpha}{2}+1} }
\int_{\mathbb{R} ^{d}} \frac{d\eta } { \vert \eta \vert ^{\alpha } \left( 1+ \frac{ \vert \eta \vert ^{4}} {4}\right)  }.
\end{eqnarray*}
Therefore, in the Riesz kernel case, $\{U(t), t\geq 0\}$ is, up to a constant, a fractional Brownian motion $ B^{\gamma}$
with Hurst parameter $\gamma= H-\frac{d-\alpha}{4}.$

Recall that the spectral density of the fBm $ B^{\gamma}$ with Hurst index $\gamma \in (0,1)$  is given by
$f_{\gamma}(\lambda)= c_{\gamma} \vert \lambda \vert ^{-(1+2\gamma)}$
with $c_{\gamma}= \frac{ \sin (\pi \gamma) \Gamma (1+2\gamma)} {2\pi }.)$ Hence, we have the following corollary.

\begin{corollary}\label{deco2}
Let $U = \left\{U(t), t\geq 0\right\}$ be the Gaussian process defined by (\ref{u1}) such that  the
spatial covariance of $W^{H}$ is given by the Riesz kernel (\ref{riesz}).Then $U$ coincides in distribution with
$C_{0} B ^{\gamma}$  with  $\gamma= H-\frac{d-\alpha}{4}$ and
$$C_{0}^2= \frac{ (2 \pi )^{-d +\frac 1 2} \alpha _{H} 2 ^{2H-1} \Gamma (H-\frac{1}{2}) }
{ \sin \left( \pi (d-\frac{H-\alpha}{4})\right) \Gamma(1+2H -\frac{d-\alpha}{2}) \Gamma (1-H) }
\int_{\mathbb{R} ^{d}} \frac{d\eta } { \vert \eta \vert ^{\alpha} \left( 1+ \frac{ \vert \eta \vert ^{4}} {4}\right)  } .$$
\end{corollary}

Now for every $t \ge 0$ we have the following decomposition
$$u(t,x)= U(t) - Y(t), $$
where
\begin{equation*}
Y(t)= \int_{-\infty} ^{0} \int_{\mathbb{R} ^{d}} \left( G(t-u, x-y) -G (-u, x-y) \right) W ^{H}(du, dy).
\end{equation*}

The following theorem shows that the sample function of $\{Y (t), t \ge  0\}$ is smooth, which is
useful for studying the regularity  properties of  the solution process $\{u(t, x), t\geq 0\}$ in the time variable.

\begin{theorem}\label{p2-notes}
Let $x \in \R^d$ be fixed and let $[a,b] \subset (0, \infty)$. Then for any $k\ge 1$ there is a modification of
$\{Y(t), t\geq 0\}$ such that its sample function is almost surely continuously differentiable on $[a, b]$ of
order $k$.
\end{theorem}

\begin{proof}:   The method of proof is similar to those of \cite[Proposition 3.1]{FK14} and  \cite[Theorem 4.8]{XX11},
but is more complicated in our fractional-colored noise case.

The mean square derivative of $Y$ at $t \in (0, \infty)$ can be expressed as
\[
Y'(t) = \int_{-\infty}^0 \int_{\mathbb{R} ^{d}} G'(t-u, x-y) \, W^{H}(du, dy),
\]
where $G' := \partial G/\partial t$. This can be verified by checking the covariance functions.
For every $s, t \in (0, \infty)$ with $s\leq t$, similarly to the
proof of Theorem \ref{deco}, we derive
\begin{eqnarray*}
\mathbb{E}\big( |Y'(t)-Y'(s) | ^{2}\big) &=& \mathbb{E} \left( \int_{-\infty} ^{0} \int_{\mathbb{R} ^{d}} 
\left( G'(t-u, x-y) -G' (s-u, x-y) \right) W ^{H}(du, dy)\right) ^{2} \\
&=&
\frac {\alpha _{H}} {4 (2 \pi) ^{d}}\, \int_{\mathbb{R}^d} \, \vert \xi \vert ^{4-\alpha} d\xi 
\int_{0} ^{\infty} \int _{0} ^{\infty} \frac{dudv}{\vert u-v\vert ^{2-2H}} \\
&&\times
\left( e^{-\frac{1}{2} (t+u)\vert \xi \vert ^{2}}- e^{-\frac{1}{2} (s+u) \vert \xi \vert ^{2}}\right)
\left( e^{-\frac{1}{2} (t+v) \vert \xi \vert ^{2}}- e^{-\frac{1}{2} (s+v)\vert \xi \vert ^{2}}\right).
\end{eqnarray*}
In the above, we have used the fact that the Fourier transform of the function $y  \mapsto G'(t+u, y)$ is
\[
\frac{\partial}{\partial t}\Big({\cal F}G(t+u, \cdot)(\xi)\Big) = - \frac 1 2 |\xi|^2 e^{-\frac{1}{2} (t+u)\vert \xi \vert ^{2}}.
\]
Denote by $\cF _{0, \infty} $ the restricted Fourier transform of $f\in L^{1} (0, \infty)$ defined by $(\cF _{0, \infty} f)(\tau )= \int_{0} ^{\infty} e ^{-ix\tau} f(x)dx$, $\tau \in \mathbb{R}$. By applying the Parseval relation (\ref{parseval}) for the  restricted transform (see Lemma A1 in \cite{BalanTudor1}), we see that for all $s, t \in [a, b]$
with $s < t$,
\begin{eqnarray*}
\mathbb{E} \big(|Y'(t)-Y'(s) | ^{2} \big)&=& C\int_{\mathbb{R}^{d} } \, \vert \xi \vert ^{4-\alpha} d\xi
\int_{\mathbb{R}} \frac{d\tau}{\vert \tau \vert ^{2H-1} } \left|   {\cal{F}}_{0, \infty} \left( e ^{-\frac{1}{2}(t+ \cdot) \vert \xi \vert ^{2}}-
e^{-\frac{1}{2}(s+ \cdot) \vert \xi \vert ^{2}}\right)(\tau)\right|  ^{2} \\
&=& C\int_{\mathbb{R}^{d} } \, \vert \xi \vert ^{4-\alpha} d\xi
\int_{\mathbb{R}} \frac{d\tau}{\vert \tau \vert ^{2H-1} } \left| e^{-\frac{1}{2} t\vert \xi \vert ^{2} }- e^{-\frac{1}{2} s\vert \xi \vert ^{2} }\right| ^{2}
\frac{1}{ \tau ^{2} + \frac{\vert \xi \vert ^{4}}{4}}\\
&=& C\int_{\mathbb{R}^{d} } \vert \xi \vert ^{4 - \alpha -4H}   e^{-s |\xi|^2}\, \left| 1- e^{-\frac{1}{2} (t-s) \vert \xi \vert ^{2}}\right| ^{2} d\xi \int_{\mathbb{R}} \frac{d \tau}{ (\vert \tau \vert ^{2} + \frac{1}{4} ) \vert \tau \vert ^{2H-1} }\\
&\leq & C \vert t-s \vert ^{2}\, \int_{\mathbb{R}^{d} } \vert \xi \vert ^{8 - \alpha -4H}   e^{-a |\xi|^2}\, d\xi.
\end{eqnarray*}
Hence, as in \cite{FK14,XX11},  by using Kolmogorov's continuity theorem, we can find a modification of $Y$ such that $Y(t)$
is continuously differentiable on $[a, b]$. Iterating this argument yields the conclusion of Theorem \ref{p2-notes}.
\end{proof}

\vskip0.2cm

\begin{remark} In \cite{MZ}, the authors obtained a similar decomposition for the solution to  the  linear heat equation with white noise in time and Riesz covariance in space in dimension $d=1$. Our Theorem \ref{deco} shows that the time behavior of the process $u$ is very similar to the behavior of the bifractional Brownian motion (see the main result in \cite{LN09}).
\end{remark}

By applying Theorems \ref{deco},  \ref{p2-notes} and the results on uniform and local moduli of continuity for Gaussian processes
(see e.g. \cite[Chapter 7]{MRbook} or \cite{MWX13}), we derive the following regularity results on the solution process
 $\{u(t, x), t \ge 0\}$, when $x\in \R^d$ is fixed.
For simplicity, we  avoid the point $t=0$.

\begin{prop}
Let $x\in \mathbb{R} ^{d} $ be fixed. Then for any $0<a<b<\infty$, we have
\begin{equation*}
\lim _{\varepsilon \to 0} \sup _{s, t \in [a, b], \vert s-t\vert \leq \varepsilon }  \frac{ \left| u(t,x)- u(s,x) \right| }{ \vert s-t\vert ^{\gamma}  \sqrt { \log (1/(t-s) }}
=c_3\ \ \  \mbox{ a.s.},
\end{equation*}
where $0<c_3< \infty$ is a constant that may depend on $a,b, \gamma $ and $x$. Or
\begin{equation*}
\limsup _{ \varepsilon \to 0} \sup _{s, t \in [a, b], \vert s-t\vert \leq \varepsilon }  \frac{ \left| u(t,x)- u(s,x) \right| }
{  \varepsilon  ^{\gamma}  \sqrt { \log (1/\varepsilon) }}  = c_3 \ \ \  \mbox{ a.s. }
\end{equation*}
Here and in Propositions \ref{prop:ILI-t} and \ref{prop:Chung-t},  $\gamma= H-\frac{d-\alpha}{4}$.
\end{prop}

\begin{prop}\label{prop:ILI-t}
Let $ x\in \mathbb{R} ^{d}$ be fixed. Then for any $t_{0} >0$ we have
\begin{equation*}
\limsup_ {\varepsilon \to 0} \frac{\sup _{\vert t-t_0\vert \leq \varepsilon }  \left| u(t,x)- u(t_{0},x) \right| }
{  \varepsilon ^{\gamma}  \sqrt {\log  \log (1/\varepsilon) }}= c_4\ \  \mbox{ a.s.},
\end{equation*}
where $0<c_4 < \infty$ is a constant.
\end{prop}

By applying Theorems \ref{deco}, \ref{p2-notes} and the Chung-type law of iterated logarithm in 
\cite{MR95},  we derive immediately

\begin{prop}\label{prop:Chung-t}
Let $ x\in \mathbb{R} ^{d}$ be fixed. Then for any $t_{0} >0$ we have
\begin{equation*}
\liminf_ {\varepsilon \to 0} \frac{\sup _{\vert t-t_0\vert \leq \varepsilon }  \left| u(t,x)- u(t_{0},x) \right| }
{  (\varepsilon /\log  \log (1/\varepsilon))^{\gamma} }= c_5\ \ \ \mbox{ a.s.},
\end{equation*}
where $0<c_5 < \infty$ is a constant depending on the small ball probability of the Gaussian process $U$
in Theorem \ref{deco}.
\end{prop}

Further properties on the local times and fractal behavior of the solution process $\{u(t, x), t \ge 0\}$,
when $x \in \R^d$ is fixed,  can be derived from \cite{Xiao96,Xiao07,Xiao1}.

\subsection{Sharp  regularity in space}
\label{sec:space}

In this section we fix $t > 0$ and analyze the space regularity of the solution $\{u(t, x), x\in \R^d\}$.
We start with the following result.

 \begin{theorem}\label{stationary}
 For each $t > 0$, the Gaussian random field $\{u(t, x), x\in \R^d\}$ is stationary with spectral measure
 \[
 \Delta (d\xi) = \alpha _{H} (2\pi) ^{-d}  \int_{0} ^{t} \int_{0} ^{t} \frac{dudv} {\vert u-v\vert ^{2-2H}} e^{-\frac{(u+v) \vert \xi \vert ^{2}}{2}}\, \mu(d\xi).
 \]
 \end{theorem}

 \begin{proof}:  It follows from the Fourier transform of the Green kernel in (\ref{Fourier-heat}) and
Parseval's identity (\ref{parseval}) that   the covariance function of $\{u(t, x), x\in \R^d\}$ is
\begin{eqnarray*}
\E \big(  u(t, x) u(t,y) \big)
& = & \frac{\alpha _{H}} {(2\pi) ^{d}}   \int_{0} ^{t} \int_{0} ^{t} \frac{dudv} {\vert u-v\vert ^{2-2H}} \int_{\mathbb{R}^{d}} e^{i \langle x-y, \xi \rangle}
e^{-\frac 1 2 (t-u)  \vert \xi \vert ^{2}} \, e^{-\frac 1 2 (t-v)  \vert \xi \vert ^{2}} \mu (d\xi) \\
& =& \frac{\alpha _{H}} {(2\pi) ^{d}}  \int_{\mathbb{R}^{d}}    e^{i \langle x-y, \xi \rangle} \bigg(\int_{0} ^{t} \int_{0} ^{t} \frac{dudv} {\vert u-v\vert ^{2-2H}}
e^{-\frac 1 2 (u+v)  \vert \xi \vert ^{2}} \bigg)\, \mu (d\xi).
\end{eqnarray*}
The conclusion of Theorem \ref{stationary} follows.
 \end{proof}

It has been shown in \cite[Proposition 4.3]{BalanTudor2} that there exist two strictly positive constants $c_{1,H}, c_{2,H}$  such that
\begin{eqnarray}
\label{n1}
c_{1,H} (t^{2H}\wedge 1) \left( \frac{1}{1+ \vert \xi \vert ^{2}} \right) ^{2H} &\leq&  \int_{0} ^{t} \int_{0} ^{t} \frac{dudv}
{ \vert u-v\vert ^{2-2H}}   e^{-\frac{(u+v) \vert \xi \vert ^{2}}{2}} \nonumber \\
&& \quad \qquad \le  c_{2,H} (t^{2H}+1) \left( \frac{1}{1+ \vert \xi \vert ^{2}} \right) ^{2H}.
\end{eqnarray}
This and Condition (\ref{hyp}) imply that the spectral measure  $\Delta (d\xi)$ is comparable with an absolutely continuous
measure with a density function that is comparable to $|\xi|^{-(\alpha + 4H)}$ for all $\xi \in \R^d$ with $|\xi| \ge 1$.
As shown in \cite{Pitt, LX12, Xiao07, Xiao1},  this information is very useful for studying regularity and other sample path
properties of the Gaussian random field $\{u(t,x), x \in \R^d\}$. In the following we show some consequences.

We start with the following estimate on $\E \big(\vert u(t,x) -u(t,y) \vert ^{2}\big)$.  To this end, let $\beta = \min \{1, 2H - \frac{d -\alpha}2 \}$,
and let
$$
\rho = \left\{\begin{array}{ll}
1\ \  &\hbox{ if } \, \beta = 1,\cr
0 &\hbox{ otherwise.}
\end{array}
\right.
$$
\begin{theorem}\label{pspace}
Assume that  (\ref{hyp}) and (\ref{cc}) hold. For any $M>0$ and $t>0$,
there exist  positive and finite constants
$c_{6}$,  $c_{7}$ such that for any $x,y \in [-M, M]^{d} $,
\begin{equation}\label{variogram-space}
c_{6} \vert x-y\vert ^{2\beta} \Big(\log \frac 1 {|x-y|}\Big)^\rho \leq
\E \big(\left| u(t,x) -u(t,y) \right| ^{2}\big) \le c_{7} \vert x-y\vert ^{2 \beta} \Big(\log \frac 1 {|x-y|}\Big)^\rho.
\end{equation}
\end{theorem}

\begin{proof}: Take $x,y \in [-M, M]^{d} $ and let  $z:=x-y\in \R^d$.
Using again  (\ref{Fourier-heat}) and Parseval's identity,  we can write
\begin{eqnarray*}
&&\E \big( \left| u(t, x+z)- u(t,x) \right| ^{2}\big)\\
&& = \alpha _{H} (2\pi) ^{-d}  \int_{0} ^{t} \int_{0} ^{t} \frac{dudv} {\vert u-v\vert ^{2-2H}} \int_{\mathbb{R}^{d}}\big| e^{-i \langle \xi, z \rangle}-1\big| ^{2}
e^{-\frac{u \vert \xi \vert ^{2}}{2}}\, e^{-\frac{v \vert \xi \vert ^{2}}{2}} \mu (d\xi) \\
&& \asymp  \int_{\mathbb{R}^{d}} (1-\cos \langle \xi, z \rangle ) \frac{d\xi}{ \vert \xi \vert ^{\alpha}}
\int_{0} ^{t} \int_{0}^{t} \frac{dudv}{ \vert u-v\vert ^{2-2H}}\,   e^{-\frac{(u+v) \vert \xi \vert ^{2}}{2}}.
\end{eqnarray*}

Let us first prove the lower bound in (\ref{variogram-space}). Using (\ref{n1}), with $c_{1,H,t}$
a generic strictly positive constant depending
on $t,H$ (that may change from line to line) and the lower bound in (\ref{cc}), we derive
\begin{eqnarray*}
\E \left| u(t, x+z)- u(t,x) \right| ^{2}&\geq&  c_{1,H,t}  \int_{\vert \xi \vert \geq 1} \frac{d\xi }
 {\vert \xi \vert ^{\alpha}} \left( \frac{1}{1+ \vert \xi \vert ^{2}} \right) ^{2H}(1-\cos \langle \xi, z \rangle )\\
&\geq &c_{1,H,t}  \int_{\vert \xi \vert \geq 1} \frac{d\xi} { \vert \xi \vert ^{\alpha +4H}} (1-\cos \langle \xi, z \rangle ).
\end{eqnarray*}
By making the change of variables using spherical coordinates,  we have
\begin{equation}\label{Eq:new1}
\E \left| u(t, x+z)- u(t,x) \right| ^{2}
\geq  c_{1,H,t}\vert z\vert ^{-d+ \alpha +4H} \int_{{\mathbb S}^{d-1}} \int_{\vert z\vert}^\infty  r^{d-1 -\alpha -4H} (1- \cos (r \langle \theta, \theta_z\rangle )
dr \sigma(d\theta),
\end{equation}
where $\theta_z =\frac{z} {\vert z\vert}$ and $\sigma(d \theta)$ is the uniform measure on  the unit sphere ${\mathbb S}^{d-1}$.

Next we distinguish three cases: (i) $2H - \frac{d -\alpha}2 >1$, (ii) $2H - \frac{d -\alpha}2 =1$ and (iii) $2H - \frac{d -\alpha}2 <1$.

In case (i) and (ii), we observe that for $|z|$ small,
\begin{eqnarray*}
\int_{\vert z\vert}^\infty  r^{d-1 -\alpha -4H} (1- \cos (r \langle \theta, \theta_z\rangle )
dr &\geq&  \frac 1 2 \int_{\vert z\vert}^1  r^{d-1 -\alpha -4H} (r \langle \theta, \theta_z\rangle )^2  dr\\
&\ge&  C \vert z \vert ^{d +2 -\alpha - 4H} \Big(\log \frac 1 {|z|}\Big)^\rho \langle \theta, \theta_z\rangle ^2,
\end{eqnarray*}
where the extra factor $\log \frac 1 {|z|}$ appears in case (ii). Plugging this into (\ref{Eq:new1}) gives the desired lower bound.

In case (iii), the integrand $r \mapsto r^{d-1 -\alpha -4H} (1- \cos (r \langle \theta, \theta_z\rangle )$ is integrable at both 0 and  infinity.
The fact that $x,y\in [-M,M] ^{d}$ ensures the integral has a positive lower bound.
This gives the lower bound in  (\ref{variogram-space}).




Now we verify the upper bound in \eqref{variogram-space}. Similarly to the above, the right-hand side of (\ref{n1}) and (\ref{cc}) imply
\begin{eqnarray*}
&&\E \big(\left| u(t, x+z)- u(t,x) \right| ^{2}\big)\\
&& \leq c_{2,H,t}    \vert z\vert ^{-d+ \alpha +4H} \int_{{\mathbb S}^{d-1}} \int_{0}^\infty  r^{d-1 -\alpha }\left( \frac{1}{\vert z\vert ^{2}+ r^{2} } \right) ^{2H}  (1- \cos (r \langle \theta, \theta_z\rangle )  dr \sigma(d\theta).
\end{eqnarray*}
Again, by distinguish three cases: (i) $2H - \frac{d -\alpha}2 >1$, (ii) $2H - \frac{d -\alpha}2 =1$ and (iii) $2H - \frac{d -\alpha}2 <1$, we can verify that the upper bound in \eqref{variogram-space}. Since this is elementary, we omit the details.
\end{proof}

Theorem \ref{pspace} suggests that the sample function $x \mapsto u(t, x)$ is rough (or fractal) when  $2H - \frac{d -\alpha}2 \le 1$,
and is differentiable when $2H - \frac{d -\alpha}2>1$. This is indeed the case as shown by the following theorem.

\begin{theorem}\label{Diff}
Assume that  (\ref{hyp}) and (\ref{cc}) hold and $t > 0$ is fixed.  If $2H - \frac{d -\alpha}2>1$, then  $\{u(t, x), x \in \R^d\}$
has a modification (still denoted by the same notation) such that almost surely the sample
function $x \mapsto u(t, x)$ is continuously differentiable on $\R^d$. Moreover, for any $M>0$,
there exists a positive positive random variable $K$ with all moments   such that for every $j = 1,  \ldots, d$,
the partial derivative $\frac{\partial}{\partial x_j} u(t,x)$ has the following modulus of continuity on $ [-M, M]^{d} $:
\begin{equation}\label{mod-derivative}
\sup_{x, y \in [-M, M]^{d}, |x-y|\le \varepsilon} \Big|\frac{\partial}{\partial x_j} u(t,x) - \frac{\partial}{\partial y_j} u(t,y) \Big|
\le K \varepsilon^{2H - \frac{d -\alpha}2-1} \sqrt{\log \frac 1 {\varepsilon}}.
\end{equation}
\end{theorem}
\begin{proof}: The method of proof is similar to that of Theorem \ref{p2-notes} above or \cite[Theorem 4.8]{XX11}. By applying Theorem \ref{stationary},
we can show that the mean square partial derivative $\frac{\partial}{\partial x_j} u(t,x)$ exists and
\begin{eqnarray*}
\E\bigg(\Big|\frac{\partial}{\partial x_j} u(t,x) - \frac{\partial}{\partial y_j} u(t,y) \Big|^2\bigg)
&=& \int_{\R^d} \xi_j^2 \big|e^{i \langle \xi, x\rangle} - e^{i \langle \xi, y\rangle}\big|^2 \Delta(d\xi)\\
&\asymp& \int_{\R^d} \xi_j^2 \big( 1- \cos  \langle \xi, x-y\rangle  \big) \Big(\frac 1 {1 + |\xi|^2}\Big)^{2H} \frac{d\xi} {|\xi|^\alpha} \\
&\le& C\, |x-y|^{4H +\alpha - d - 2}.
\end{eqnarray*}
In the above, we have used the fact that $0< 4H +\alpha - d - 2 < 2$, so the last inequality can be derived as in the case of fractional Brownian motion.
Finally the desired result follows from the general results on modulus of continuity of   Gaussian random fields (cf. \cite[Chapter 7]{MRbook}
or \cite[Section 4]{Xiao1}).
\end{proof}


Finally, we consider the non-smooth case. For simplicity, we assume  $2H - \frac{d -\alpha}2 <1$. The case $2H - \frac{d -\alpha}2 =1$
is more subtle and will require significant extra work.

By combining Theorem \ref{stationary} and relation (\ref{n1})  with the results in \cite[Section 8]{Pitt} (see also \cite{LX12, Xiao07}),
we obtain the following useful lemma.

\begin{lemma} \label{prop:SLND}
Suppose $ 2H-\frac{d-\alpha}{2} < 1$. Then ,for every fixed $t>0$, the Gaussian field $\{u(t,x), x \in \mathbb{R} ^{d} \}$
is strongly locally nondeterministic. Namely,
for every $M>0$, there exists a constant $c_{8}>0$ (depending on $t$ and $M$) such that for every $n\geq 1$ and
for every $ x, y_{1},..., y_{n} \in [-M,M] ^{d} $,
\begin{equation*}
{\rm Var}  \left( u(t,x) | u(t, y_{1}),\ldots, u(t, y_{n} ) \right) \geq c_{8} \min _{ 0\leq j\leq n} \{ \vert x -y_{j} \vert^{ 4H+\alpha -d} \},
\end{equation*}
where $y_{0}=0$.
\end{lemma}

Because of this, the Gaussian random field $\{u(t,x), x \in \mathbb{R} ^{d} \}$ shares many local properties with fractional
Brownian motion $B^{\beta} = \{B^\beta (x), x \in \R^d\}$ with $\beta = 2H-\frac{d-\alpha}{2}$.
We give some examples.

\begin{prop}
\label{prop:moduli-space}
(Uniform and local moduli of continuity)
Suppose $ 2H-\frac{d-\alpha}{2} < 1$. Let $t > 0$ and $M>0$ be fixed. Then
\begin{itemize}
\item[(a)]\, almost surely,
\begin{equation*}
\lim_{\varepsilon \to 0} \frac { \max _{x\in [-M, M] ^{d}, \vert h\vert \leq \varepsilon } |u(t, x+h)- u(t, x)|}
{ \varepsilon^{\beta} \sqrt{ \log (1/\varepsilon) }}= c_{9}.
\end{equation*}
\item[(b)]\, For $x_{0} \in \mathbb{R}^{d}$,
\begin{equation*}
\limsup _{\varepsilon \to 0} \frac{ \max_ {\vert h\vert \leq \varepsilon } \vert u( t, x_{0}+h)- u(t, x_{0}) \vert }
{\varepsilon ^{\beta} \sqrt{\log \log (1/\varepsilon )} }= c_{10}.
\end{equation*}
\end{itemize}
In the above, $\beta= 2H-\frac{d-\alpha}{2}$ and $0< c_{9}, c_{10} < \infty$ are constants.
\end{prop}
\begin{proof}: The conclusion (a) follows from Lemma \ref{prop:SLND} and  \cite[Theorem 4.1]{MWX13},
while (b) follows from \cite[Theorem 5.1]{MWX13}.
\end{proof}

\vskip0.3cm

\begin{prop} \label{prop:Chung-space}
(Chung's LIL)
Suppose $ 2H-\frac{d-\alpha}{2} < 1$. Then for every $t > 0$ and $x_{0} \in \mathbb{R}^{d}$, we have
\begin{equation*}
\liminf _{\varepsilon \to 0} \frac{ \max_ {\vert h\vert \leq \varepsilon } \vert u( t, x_{0}+h)- u(t, x_{0}) \vert }
{\varepsilon ^{\beta} / (\log \log (1/\varepsilon )^{ \beta }}= c_{10},
\end{equation*}
where $\beta= 2H-\frac{d-\alpha}{2}$ and $0< c_{10} < \infty$  is a constant.
\end{prop}
\begin{proof}: It follows from Lemma \ref{prop:SLND} and Li and Shao \cite{LiShao01}. See
also \cite[Theorem 3.2 ]{LX10} for related results.
\end{proof}

\vskip0.2cm

\b We conclude with the following remarks.
\begin{remark}
\begin{itemize}
\item[(i)]\,
Propositions \ref{prop:moduli-space} and \ref{prop:Chung-space} show that, when $\beta = 2H-\frac{d-\alpha}{2} < 1$ and
$t>0$ is fixed,  the solution process
$\{u(t, x), x \in \R^d\}$ behaves locally like a fractional Brownian motion $\{B^\beta(x), x \in \R^d\}$. This is not surprising.
Indeed, similarly to Theorem \ref{p2-notes} and \cite[Proposition 3.1]{FK14}, one can show that $\{u(t, x), x \in \R^d\}$ is
a smooth perturbation of $\{B^\beta(x), x \in \R^d\}$.
\item[(ii)]\,
Further properties on the local times and fractal behavior of the solution process $\{u(t, x), x \in \R^d\}$, when $t >0$ is fixed,
can be derived from \cite{Xiao96,Xiao07,Xiao1}. It is also possible to investigate sample path properties of the
Gaussian random field  $\{u(t, x), t \ge 0, x \in \R^d\}$ in both time and space variables. The methods developed
 in \cite{DMX14} will be useful for this purpose. We will study these problems in a subsequent paper.
\end{itemize}
\end{remark}

\par\bigskip\noindent
{\bf Acknowledgment.} Yimin Xiao thanks Professor  Raluca M. Balan for stimulating discussions.


\begin{thebibliography}{99}

\bibitem{B}
{R. M. Balan (2012): } {\em  Some linear SPDEs driven by a fractional noise with Hurst index greater than $1/2$. } Infin. Dimens. Anal. Quantum Probab. Relat. Top. {\bf 15} (4).

\bibitem{BC}
{R.M. Balan and D. Conus (2014): }{\em A note on intermittency for the fractional heat equation. }Statist. Probab. Lett. {\bf 95}, 6--14.


\bibitem{BalanTudor1}
{R.M.  Balan and  C.A. Tudor (2008): } {\em The stochastic heat
equation with fractional-colored noise: existence of the solution.}
Latin Amer. J. Probab. Math. Stat. {\bf 4}, 57--87.

\bibitem{BalanTudor2}
{R.M. Balan  and C. A. Tudor  (2010): }
{\em The stochastic wave equation with fractional noise: A random field approach.}
Stoch. Process. Appl. {\bf 120}, 2468--2494.


\bibitem{BPS04}
{E. Bayraktar, V. Poor and R. Sircar(2004): } {\em
Estimating the fractal dimension of the SP 500 index using wavelet
analysis. } Intern. J. Theor. Appl. Finance,  {\bf 7}, 615--643.





\bibitem{BoTu}
{S. Bourguin and C.A. Tudor (2011): } {\em On the law of the
solution to a stochastic heat equation with fractional noise in
time.} Preprint, to appear in {\em Random Operators and Stochastic Equations. }


\bibitem{CCL03}
{D. del Castillo-Negrete, B.A. Carreras and V.E. Lynch (2003): }
{\em  Front dynamics in reaction-diffusion systems with Levy
flights: a fractional difussion approach. }  Phys. Rev. Letters,
{\bf 91} (1), 018302.




\bibitem{Da}
{R.C. Dalang (1999): } {\em Extending martingale
measure stochastic integral with applications to spatially
homogeneous s.p.d.e.'s. } Electr. J. Probab. {\bf 4}, no. 6, 29
pp. Erratum in  Electr. J. Probab. {\bf 6} (2001), 5 pp.

\bibitem{DMX14}
{R.C. Dalang, C. Mueller and Y. Xiao (2014): }  {\em Local properties of the solution of stochastic heat equation
with spatially colored noise.} In preparation.

\bibitem{DaSa2}
{R.C. Dalang and M. Sanz-Sol\'e (2009): } {\em H\"older-Sobolev regularity of the solution to
the stochastic wave equation in dimension three.} Memo. Amer. Math. Soc.  {\bf 199}, no. 931.


\bibitem{DMS03}
{G. Denk, D. Meintrup and S. Schaffer (2004): }
{\em Modeling, simulation and optimization of integrated circuits. }
Intern. Ser. Numerical Math. {\bf 146}, 251--267.

\bibitem{FK14}
{M. Foondun and D. Khoshnevisan (2014):} {\it Analysis of the gradient of the solution to a
stochastic heat equation via fractional Brownian motion.} Stoch. PDE: Anal. Comp., to appear.


\bibitem{GLT06}
{M. Gubinelli, A. Lejay  and S. Tindel S. (2006): }
{\em  Young integrals and SPDE's.} Potential Anal. {\bf 25}, 307--326.

\bibitem{HV}
{C. Houdr\'e and J. Villa (2003): } {\em  An example of infinite dimensional quasi-helix. }
Stochastic Models, Contemporary Mathematics {\bf 366}, 195--201.




\bibitem{KouSinney04}
{S. C. Kou and X. Sunney Xie (2004): } {\em  Generalized
Langevin equation with fractional Gaussian noise: subdiffusion
within a single protein molecule.} Phys. Rev. Letters, {\bf93}(18), 180603.

\bibitem{LN09}
{P. Lei and D. Nualart (2009): }
{\em A decomposition of the bifractional Brownian motion and some applications.}
        Statist. Probab. Lett. {\bf 79},  619--624.

\bibitem{LiShao01}
W. V. Li and Q.-M. Shao (2001):  {\em Gaussian processes: inequalities,
small ball probabilities and applications.}  Stochastic
Processes: Theory and Methods. Handbook of Statistics,
    {\bf 19}, (C. R. Rao and D. Shanbhag, editors), pp. 533--597,
North-Holland.

\bibitem{LX10}
{N. Luan and Y. Xiao (2010): } {\em Chung's law of the iterated logarithm
for anisotropic Gaussian random fields.}  Statist. Probab. Lett. {\bf 80}, 1886--1895.

\bibitem{LX12}
{N. Luan and Y. Xiao (2012): }  {\em Spectral conditions for strong local nondeterminism
and exact Hausdorff measure of ranges of Gaussian random fields.}
J. Fourier Anal. Appl. {\bf 18}, 118--145.


\bibitem{MRbook}
M.B. Marcus and J. Rosen (2006): Markov Processes, Gaussian Processes, and Local Times.
Cambridge University Press, Cambridge.

\bibitem{maslowski-nualart03}
{B. Maslovski and D. Nualart (2003): } {\em Evolution equations driven by a
fractional Brownian motion.} J. Funct. Anal. {\bf 202}, 277-305.

\bibitem{MWX13}
{M.M. Meerschaert, W. Wang and Y. Xiao (2013): } {\em Fernique-type inequalities
and moduli of continuity of anisotropic Gaussian random fields.}  Trans.
Amer. Math. Soc. {\bf 365}, 1081--1107.

\bibitem{MR95}
{D. Monrad and H. Rootz\'en (1995): } {\em  Small values of Gaussian processes and
functional laws of the iterated logarithm.}  Probab. Theory Relat. Fields {\bf 101},
173--192.

\bibitem{MT02}
{C. Mueller and R. Tribe (2002): } {\em Hitting probabilities of a random
string.}  Electronic J. Probab. {\bf 7}, Paper No. 10, 29 pp.

\bibitem{MZ}
{C. Mueller and Z. Wu (2009): } {\em A connection between the stochastic heat equation
and fractional Brownian motion, and a simple proof of a result of Talagrand. }
Electronic Comm. Probab. {\bf 14}(6), 55--65.



\bibitem{nualart-vuillermont06}
{D. Nualart and P.-A. Vuillermont (2006): } {\em
Variational solutions for partial differential equations driven by fractional a noise.}
J. Funct. Anal. {\bf 232}, 390--454.


\bibitem{EF}
{E. Nualart  and F. Viens  (2009): } {\em The fractional stochastic heat equation
on the circle: time regularity and potential theory.} Stoch. Proc. Appl.
{\bf 119}, 1505--1540.


\bibitem{OT}
{H. Ouahhabi and C.A. Tudor (2012): } {\em Additive functionals of the solution to fractional
stochastic heat equation.}  J. Fourier Anal. Appl. {\bf 19} (4), 777-791.

\bibitem{Pitt}
{L. D. Pitt (1978): } {\em Local times for Gaussian vector fields. } Indiana Univ. Math.
J. {\bf 27}, 309--330.

\bibitem{QS-tindel07}
{L. Quer-Sardanyons and S. Tindel (2007): }
{\em The 1-d stochastic wave equation driven by a fractional Brownian sheet.} Stoch. Proc.
Appl. {\bf 117}, 1448--1472.

\bibitem{RT}
{F. Russo and C.A. Tudor (2006): } {\em On bifractional Brownian
motion.} Stoch. Proc. Appl.  {\bf 116}, 830--856.

\bibitem{sanzsole-vuillermont07}
{M. Sanz-Sol\'e and P.-A. Vuillermot  (2009):} {\em Mild solutions for a class of fractional
SPDE's and their sample paths.} J. Evolution Equations {\bf 9}, 235--265.

\bibitem{Swanson1}
{J. Swanson (2007): } {\em Variations of the solution to a stochastic
heat equation.} Ann. Probab.  {\bf 35}, 2122--2159.


\bibitem{TTV}
{S. Torres, C.A. Tudor and F. Viens (2014): }  {\em Quadratic variations for the fractional-colored stochastic heat equation. } Electron. J. Probab. {\bf 19} (76), 51 pp.

\bibitem{T}
{C.A. Tudor (2013): } Analysis of variations for self-similar processes. A stochastic calculus approach. Probability and its Applications (New York). Springer, Cham.

\bibitem{TX}
{C.A. Tudor and Y. Xiao (2007): }  {\em  Sample path properties of bifractional
Brownian motion.} Bernoulli {\bf 13}, 1023--1052.

\bibitem{WX06}
{D. Wu and Y. Xiao (2006): }  {\em Fractal properties of the random string
processes.} In:  IMS Lecture Notes-Monograph Series--High
Dimensional Probability, {\bf 51}, 128--147. Institute of
Mathematical Statistics, Beachwood, Ohio, U.S.A.


\bibitem{Xiao96}
{Y. Xiao (1996): } {\em H\"older conditions for the local times and the
Hausdorff measure of the level sets of Gaussian random fields.}
 Probab. Th. Rel. Fields  {\bf 109}, 129--157.

\bibitem{Xiao07}
{Y. Xiao (2007): }  {\em Strong local nondeterminism and the sample path
properties of Gaussian random fields.} In: Asymptotic Theory in
Probability and Statistics with Applications (Tze Leung Lai, Qiman
Shao, Lianfen Qian, editors), pp. 136--176, Higher Education Press,
Beijing.

\bibitem{Xiao1}
{Y. Xiao (2009): } {\em Sample path  properties of anisotropic Gaussian random fields. }
A Minicourse on Stochastic Partial Differential Equations, pp. 145--212, Lecture Notes in Math., 1962,
Springer, Berlin.

\bibitem{XX11}
{Y. Xue and Y. Xiao (2011): }  {\em Fractal and smoothness properties of space-time Gaussian models.}
Frontiers Math. China {\bf 6}, 1217--1246.


\end{thebibliography}
\end{document}